\documentclass{amsart}
\usepackage[normalem]{ulem}
\usepackage{luatex85}
\numberwithin{equation}{section}
\usepackage{xcolor}
\usepackage{pgffor}
\usepackage{url} 
\usepackage{breakurl}

\usepackage{tikz}

\definecolor{Vino}{rgb}{0.356,0,0}
\definecolor{Ruta}{rgb}{0.309, 0.459, 0.208}
\definecolor{Modra}{rgb}{0,0,0.356}
\definecolor{Rumena}{rgb}{0.5,0.35,0}
\usepackage[final,colorlinks,linkcolor=Rumena,citecolor=Modra,urlcolor=black,breaklinks]{hyperref}

\usepackage{longtable}
\usepackage[all]{xy}
\usepackage{amssymb}
\usepackage{comment}
\usepackage{stmaryrd}
\usepackage{xcolor}
\usepackage{mathtools}
\usepackage{enumitem}
\usepackage[T1]{fontenc}

\definecolor{Vino}{rgb}{0.356,0,0}
\definecolor{Ruta}{rgb}{0.309, 0.459, 0.208}
\let\cal\mathcal

\def\Oscr{{\cal O}}

\let\blb\mathbb
\def\CC{{\blb C}}

\def \PP{{\blb P}}

\def \ZZ{{\blb Z}}

\def \NN{{\blb N}}
\def \RR{{\blb R}}

\def\Ann{\operatorname{Ann}}

\def\Supp{\mathop{\text{\upshape Supp}}}

\def\Spec{\operatorname {Spec}}

\def\im{\operatorname {im}}

\def\ker{\operatorname {ker}}

\def\r{\rightarrow}

\DeclareMathOperator{\Proj}{Proj}

\def\codim{\operatorname{codim}}

\newtheorem{lemma}{Lemma}[section]
\newtheorem{proposition}[lemma]{Proposition}

\newtheorem{corollary}[lemma]{Corollary}

\theoremstyle{definition}

{

}

\theoremstyle{remark}

\newtheorem{remark}[lemma]{Remark}

\newdimen\uboxsep \uboxsep=1ex
\def\uboxn#1{\vtop to 0pt{\hrule height 0pt depth 0pt\vskip\uboxsep
\hbox to 0pt{\hss #1\hss}\vss}}

\def\uboxs#1{\vbox to 0pt{\vss\hbox to 0pt{\hss #1\hss}
\vskip\uboxsep\hrule height 0pt depth 0pt}}

\let\oldmarginpar\marginpar
\long\def\marginpar#1{\oldmarginpar{\raggedright \tiny \baselineskip 0pt #1}}

\makeatletter
\newcommand*\bigcdot{\mathpalette\bigcdot@{.5}}
\newcommand*\bigcdot@[2]{\mathbin{\vcenter{\hbox{\scalebox{#2}{$\m@th#1\bullet$}}}}}
\makeatother

\makeatletter
\newcommand{\pushright}[1]{\ifmeasuring@#1\else\omit\hfill$\displaystyle#1$\fi\ignorespaces}
\newcommand{\pushleft}[1]{\ifmeasuring@#1\else\omit$\displaystyle#1$\hfill\fi\ignorespaces}
\makeatother

\author[\v{S}pela \v{S}penko and Michel Van den Bergh]{\v{S}pela \v{S}penko and Michel
  Van den Bergh} 
\address[\v{S}pela \v{S}penko]{D\'epartement de Math\'ematique, Universit\'e Libre de Bruxelles, Campus de la Plaine CP 213, Bld du Triomphe, B-1050 Bruxelles}
\email{spela.spenko@ulb.be}
\address[Michel Van den Bergh]{Vakgroep Wiskunde, Universiteit Hasselt, Universitaire Campus \\
  B-3590 Diepenbeek}
\email{michel.vandenbergh@uhasselt.be}
\address[Michel Van den Bergh]{Vakgroep Wiskunde en Data Science, Vrije Universiteit Brussel, Pleinlaan 2, 1050 Brussel} 
\email{michel.van.den.bergh@vub.be}
\thanks{The first author is supported by a MIS grant from the National Fund for Scientific Research (FNRS) and an ACR grant from the Université Libre de Bruxelles. The second author is a senior researcher at the Research
  Foundation Flanders (FWO).  This project has received funding from the European Research Council (ERC) under the European Union's Horizon 2020 research and innovation programme (grant agreement No 885203). The project was also supported by  FWO grant G0D8616N: ``Hochschild cohomology and
  deformation theory of triangulated categories''.}

\title[On the GKZ discriminant locus]{On the GKZ discriminant locus}

\begin{document}  

\begin{abstract}
Let $A$ be an integral matrix and let $P$ be the convex hull of its columns. By a result of Gelfand, Kapranov and Zelevinski, the so-called principal $A$-determinant locus is equal to  the union of the \emph{closures} of the discriminant loci of the Laurent polynomials associated
  to the faces of $P$ \emph{that are hypersurfaces}. In this short note we show that it is also the straightforward union of all the discriminant loci,
  i.e.\ we may include those
  of higher codimension, and there is no need to take closures. This
  answers a question by Kite and Segal.
\end{abstract}

\maketitle

\section{Introduction and statement of the main result}
Let $k,d\in \NN$ and let $A=\{a_1,\dots,a_d\}\subset \ZZ^k$.  
 Let $P\subset\RR^{k}$ be the convex hull of $(a_i)_i$.
Let $F$ be a face  of $P$.
Denote by $\nabla_F\subset \CC^{F\cap A}$ the {\em discriminant
  locus}\footnote{In \cite[\S9.1.A]{GKZbook} $\nabla_F$ is defined as the 
  closure of our $\nabla_F$.} of
$f^F:=\sum_{a_i\in F} \alpha_i x^{a_i}\in \CC[x_1^{\pm 1},\ldots,x_k^{\pm 1}]$, i.e.\ the set of all
$\alpha\in \CC^{F\cap A}$ such that the singular locus of $V(f_\alpha^F)$ (i.e.\ the
common zeros of $f_{\alpha}^F$, $(\partial_i f_{\alpha}^F)_{i=1,\ldots,k}$) meets $(\CC^*)^k$
\cite[\S9.1.A]{GKZbook}. 
It is known that $\bar{\nabla}_F$ is irreducible and $\dim \bar{\nabla}_F<|F\cap A|$ \cite[Chapter 9, after (1.1)]{GKZbook}. Set 
$V(A):=\bigcup_F p^{-1}_F(\nabla_F)$,  where $p_F:\CC^A\to \CC^{F\cap A}$ is the projection. 
If $\bar{\nabla}_F$ is a hypersurface in $\CC^{F\cap A}$ (i.e.\ if $\dim \nabla_F=|F\cap A|-1$) then we denote by $\Delta_F$ its defining equation,  otherwise we set $\Delta_F=1$.
We denote the composition $\Delta_F\circ p_F:\CC^A\r \CC$ also by $\Delta_F$.
The {\em principal $A$-determinant} is a certain polynomial function on $\CC^A$
which is defined in \cite[\S10.1.A]{GKZbook}.
In \cite[Theorem 10.1.2]{GKZbook} it is shown that $E_A=\prod_{F\subseteq P}\Delta_F^{m_F}$ for suitable multiplicities $m_F\geq 1$.
Hence
\begin{equation}
\label{basic_inclusion}
V(E_A)=\bigcup_{F\subset P} V(\Delta_F)\subset  \bigcup_{F\subset P}\overline{p^{-1}_F(\nabla_F)}=\overline{V(A)}
\end{equation}
and furthermore $V(A)$ and $V(E_A)$  only differ in codimension $\leq 2$. 
We show that they in fact coincide and moreover $V(A)=\overline{V(A)}$, i.e.\ $V(A)$ is automatically closed.
This  answers positively a question by Kite and Segal in \cite[p.21]{Kitethesis}, \cite[Remark 4.9]{KiteSegal}.
\begin{proposition}\label{lem:VEA=VA}
We have $V(A)=\overline{V(A)}=V(E_A)$. 
\end{proposition}
\begin{remark}
Note that the individual parts $p_F^{-1}(\Delta_F)$ in the definition of $V(A)$ are in general not locally closed (so they are not subschemes). We can already see this in the simple example $A=\{(2,1),(1,1),\allowbreak (0,1)\}$. In that case we have $f^P(x,y)=ax^2y+bxy+cy$
for $x,y\in \CC^\ast$ and $a,b,c\in \CC$. The parts in the decomposition of $V(A)$ are:
\[
V(A)=\biggl(\biggl(V(b^2-4ac)-V(b,a)-V(b,c)\biggr)\cup V(a,b,c)\biggr)\cup V(a)\cup V(c)
\]
which fit together to form the hypersurface $V(ac(b^2-4ac))$ as predicted by Proposition \ref{lem:VEA=VA}.
The parts $p_F^{-1}(\Delta_F)$ may also have higher codimension; this occurs for example in the so-called ``quasi-symmetric'' case for all $F$ such that $\{a_i\mid a_i\in F\}$ do not form a circuit, cf. \cite[Definition 3.2, Theorem 3.4]{Kite}. 
\end{remark}
\begin{remark}
 Our interest in Proposition \ref{lem:VEA=VA} comes from the fact that it implies that if  $\alpha\not\in V(E_A)$ then the zero
  fiber of $f_{\alpha}^F:(\CC^*)^k\to \CC$ is nonsingular for all $F$,
  which does not follow directly from the description of $V(E_A)$. 
This observation is used crucially in  \cite{main_document} and seems to be of independent interest.
\end{remark}

\section{Toric preliminaries}

Let $X_A$ be the closure of $\{(x^{a_1}:\dots:x^{a_d})\in \PP^{d-1}\mid x\in (\CC^*)^k\}$ in $\PP^{d-1}$. 
The action of $(\CC^\ast)^k$ on $\PP^{d-1}$ with weights $a_1,\ldots,a_d$ restricts to an action on $X_A$
which has a dense orbit. This makes $X_A$ into a toric variety (generally non-normal).
\begin{proposition}\cite[Proposition 5.1.9]{GKZbook}\label{prop:XA}
The  torus orbits in $X_A$ are in $1{:}1$ correspondence with the faces
  of $P$. The orbit $O_F$ corresponding to a face $F$ of $P$ consists of
  $\{(z_1:\dots:z_d)\in X_A\mid z_i\neq 0\; \text{if $a_i\in F$},
  z_i=0 \;\text{if $a_i\not\in F$}\}$. $O_F$ is also equal to $\{(z_1:\dots:z_d)\in
  \PP^{d-1}\mid \exists u\in (\CC^*)^k\text{such that} \;z_i=u^{a_i}\; \text{if
    $a_i\in F$}\text{ and } z_i=0\;\text{if $a_i\not\in F$}\}$.
\end{proposition}

\begin{proof}
The first two claims are proved in loc. cit. 
  As the last claim does not seem to appear explicitly in the
  given reference, we provide a proof here, assuming the other claims. We first claim that the point
  $z_c:=(z_1:\dots:z_d)$
  with $z_i=1$ if $a_i\in F$ and $z_i=0$ if $a_i\not \in F$ is contained in $O_F$.  We first prove $z_c\in X_A$.
 Let $\lambda\in \ZZ^k$, $c\in \ZZ$ be
  such that $\langle\lambda,a_i\rangle>c$ if $a_i\not\in F$ and
  $\langle\lambda,a_i\rangle=c$ if $a_i\in F$ (such $\lambda$ exists because $F$ is a face of the convex hull of $A$). Let $(y_t)_t$ be the
  $1$-parameter subgroup of $(\CC^*)^k$ defined by~$\lambda$,
  i.e. $y_t=(t^{\lambda_1},\dots,t^{\lambda_k})$. Put $z_0:=(1:\dots:1) \in X_A$. Then
  $\lim_{t\to 0} y_t\cdot z_0=\lim_{t\to 0} t^{-c}y_t\cdot z_0=\lim_{t\to 0}(t^{\langle \lambda,a_1\rangle-c}:\dots:t^{\langle\lambda,a_d\rangle-c})=z_c$ implies $z_c\in X_A$.
Now the description of $O_F$ given in the statement of the proposition implies $z_c\in O_F$.

  If $z$ is an arbitrary point in $O_F$ then since $z_c\in O_F$, there exists   $u\in (\CC^{\ast})^k$
such that $z=u\cdot z_c=(u^{a_1}(z_c)_1:\cdots :u^{a_k} (z_c)_d)$.
  Then $z_i=u^{a_i}$ if
    $a_i\in F$ and $z_i=0$ if $a_i\not\in F$, as asserted in the second equality.
  \end{proof}
  We also recall the following description of $X_A$.

\begin{proposition}\cite[Theorem 5.2.3]{GKZbook}\label{prop:XA1} Assume $\forall i:(a_i)_k=1$. Let $R=\CC[(x^{a_i})_i]\subset \CC[(x_i^{\pm 1})_i]$. We consider $R$ as an $\NN$-graded ring via the $x_k$-degree. In particular
  $R$ is generated by the elements of degree one given by the
  $(x^{a_i})_i$. Then $\Spec R$ is the cone over $X_A$ (and hence also $X_A=\Proj R$). In other words,
  the closed points
  $\delta:R\r \CC$ of $\Spec R$, not corresponding to the graded maximal ideal of
  $R$, are given by $x^{a_i}\mapsto w_i$ where
  $(w_1:\cdots:w_d)$ is a closed point in~$X_A$.
\end{proposition}

\section{Proof of Proposition \ref{lem:VEA=VA}}
\label{sec:proof}
 By \cite[Proposition 9.1.4]{GKZbook} discriminants are invariant under injective affine transformations $\ZZ^k\r \ZZ^{k'}$. 
Hence by first replacing $\ZZ^k$ by the affine span of $A$, and subsequently embedding $\ZZ^k$ in $\ZZ^{k+1}$ as $\ZZ^{k}\times \{1\}$
it follows that we can restrict to the following situation: 
\begin{itemize}
\item the last coordinates of the $(a_i)_i$ are all $1$;
\item the $a_i$ linearly span $\RR^k$.
\end{itemize}
\emph{We will make these assumptions in this section.}
\subsection{Discriminant locus and finiteness}
We give a characterisation of the discriminant locus $V(A)\subset \CC^A$ in terms of finiteness of certain maps.
\begin{lemma}\label{lem:folklore}
For $\alpha\in \CC^A$ set $f_\alpha=\sum_{i=1}^d \alpha_i x^{a_i}$. 
Let $R=\CC[(x^{a_i})_i]$ (like in Proposition \ref{prop:XA1}). Then $\alpha\not\in V(A)$ if and only if $\gamma:\CC[y_1,\dots,y_k]\to R$, $y_i\mapsto x_i\partial_i f_{\alpha}$ 
makes~$R$ into a finitely generated $\CC[y_1,\ldots,y_k]$-module.
\end{lemma}

\begin{proof}
   The map $\gamma$ is graded if we put $\deg y_i=1$. Put $\bar{R}=R/(\gamma(y_1),\ldots,\gamma(y_k))$.
  Then $\gamma$ is finite if and only if $\bar{R}$ is finite dimensional, by the graded version of Nakayama's lemma (see Lemma \ref{lem:fg1} below).
  
  Assume first that $\gamma$ is not finite so that $\bar{R}$ is infinite dimensional and hence $\Spec \bar{R}\neq \{\bar{m}\}$, where $\bar{m}$ is the
  graded maximal ideal of $\bar{R}$.
  Let $\delta:R\r \CC$ be a  ring homomorphism representing a closed point in $\Spec \bar{R}$, different from  $\bar{m}$.
  By construction $(\delta\gamma)(y_i)=0$, $1\leq i\leq k$. In other words $\delta(x_i \partial_i f_{\alpha})=0$ for $1\leq i\leq k$.

Moreover, by Proposition \ref{prop:XA1}, $w:=(\delta(x^{a_1}):\cdots:\delta(x^{a_d}))\in X_A$.
  Let $F$ be the convex hull of $\{a_i\mid \delta(x^{a_i})\neq 0\}$. 
By Proposition \ref{prop:XA}, $F$ is a face of $P$ and 
there exists $u\in (\CC^*)^k$ such that $u^{a_i}=\delta(x^{a_i})$ for $a_i\in F$. By definition of $F$ and $f_\alpha^F$:  $\delta(x_i \partial_i f_{\alpha}^F)=0$  for $1\leq i\leq k$.
Then
$u_i(\partial_i f_{\alpha}^F)(u)=(x_i\partial_i f_{\alpha}^F)(u)=\delta(x_i\partial_i f_{\alpha}^F)=0$ for $i=1,\ldots,k$. Since $u_i\neq 0$ we obtain $(\partial_i f_{\alpha}^F)(u)=0$. Applying this for $i=k$ 
(taking into account that $a_{i,k}=1$ for all $i$) we also get $f_{\alpha}^F(u)=0$. 
Thus, $\alpha\in p^{-1}_F(\nabla_F)\subset V(A)$. 

On the other hand, if $\alpha\in V(A)=\bigcup_F
p_F^{-1}(\nabla_F)$
then there exists a face $F$ of $P$, $u\in (\CC^*)^k$ such that
$(\partial_i f_{\alpha}^F)(u)=0$ for $i=1,\ldots,k$.
By Proposition \ref{prop:XA}  the point~$w$ with homogeneous coordinates $w_i=u^{a_i}$ if $a_i\in F$ and $w_i=0$ otherwise is contained in $X_A$.
 The point with the same (affine) coordinates in the cone over $X_A$ corresponds  by Proposition \ref{prop:XA1} 
to a ring homomorphism $\delta:R\to \CC$ which sends $x^{a_i}$ to $w_i$ for $i=1,\ldots,d$, i.e. $x^{a_i}$  
  to $u^{a_i}$ 
if $a_i\in F$ and to $0$ otherwise.
It follows  that $(\delta\gamma)(y_i)=0$, $1\leq i\leq k$. 
On the other hand $(w_i)_i\neq (0,\ldots,0)$ since $(w_i)_i$ corresponds to a point in $\PP^{d-1}$.
Hence $\Spec \bar{R}$ has a closed point $\ker \delta$ different from $\bar{m}$ and hence $\bar{R}$ must be infinite
dimensional.
\end{proof}
We now introduce some notations. We let $S_0=\CC[(\alpha_i)_i]$ (so here we view the $(\alpha_i)_i$ as variables, whereas before they were scalars), $S=S_0[y_1,\dots,y_k]$,
$T=S_0[(x^{a_i})_i]\subset S_0[(x_i^{\pm 1})_i]$. We will consider the ring homomorphism
$\gamma: S\r T:y_i\mapsto x_i \partial_i f$.
\begin{corollary}\label{cor:folkolore}
We have $V(A)=\overline{V(A)}$.
\end{corollary}

\begin{proof}
  We apply Proposition \ref{prop:finite} with $\gamma:S\r T$ as above and $W=\Spec(S_0/m)$ where $m$ corresponds to a maximal
  ideal in $S_0$ (these are parametrized by the closed points in $\CC^A$).
 It follows that the set of $\alpha\in \CC^A$ for which~$\gamma$, specialized at $\alpha$, is finite is the open set $U\subset \Spec S_0$ as in the statement of  Proposition \ref{prop:finite}. On the other hand, by Lemma \ref{lem:folklore} this set is identified with the complement of $V(A)$. Thus, $V(A)$ is closed, so it equals its closure $\overline{V(A)}$.
\end{proof}

\subsection{Proof of Proposition \ref{lem:VEA=VA}}\label{subsec:proof} 
Put $X=\Spec (S_0)_{E_A}=\CC^A\setminus V(E_A)$, $Y=\Spec S_{E_A}$, $Z=\Spec T_{E_A}$.
Let $U'$ be the open subset $U$ of $X$ from Proposition \ref{prop:finite} below where $S,T$ in the statement of the proposition stand for $S_{E_A}$, $T_{E_A}$ as defined here. By
Lemma \ref{lem:folklore}, $U'=(\CC^A\setminus V(E_A))\setminus V(A)$.
By Corollary \ref{cor:folkolore} we have $V(A)=\overline{V(A)}$. Hence
by \eqref{basic_inclusion} $V(E_A)\subset V(A)$. We conclude  that  $U'=\CC^A\setminus V(A)$.  
Since for all $F$, $\dim \bar{\nabla}_F<|F|$ \cite[Chapter 9, after
  (1.1)]{GKZbook}
we have
  $U'\neq\emptyset$ so that $U'\subset X$ is dense open.

By Proposition \ref{prop:finite}, 
$\gamma_{U'}:Z\times_X U'\r Y\times_X U'$ is finite. As $Y,Z$ are
irreducible and have the same dimension (using our hypothesis in this
section that $\RR A=\RR^k$), this implies that $\gamma_{U'}$ is
dominant.

Hence $\gamma_{U',\ast}\Oscr_{Z\times_X U'}$ is a torsion free (since $Z\times_X U'$ is integral and $\gamma_{U'}$ is dominant), coherent $\Oscr_{Y\times_X U'}$-module. Let $j:Y\times_X U'\r Y$ be the injection.
Since $Y\to X$ is flat and $\codim (X-U',X)\ge 2$ (because of the fact that $V(A)$ and $V(E_A)$ only differ in codimension two) it follows that  $\codim(Y-Y\times_X U',Y)\ge 2$. 
By the Grothendieck-Serre theorem \cite[Th\'eor\`eme 1]{MR212214} $j_\ast\gamma_{U',\ast} \Oscr_{Z\times_X U'}$ is a coherent $\Oscr_{Y}$-module.
Since $\gamma_\ast \Oscr_Z$ is contained
in $j_\ast \gamma_{U',\ast} \Oscr_{Z\times_X U'}$ the former is a coherent $\Oscr_Y$-module as well. Hence $\gamma:Z\r Y$ is finite. 
Hence by the definition of $U'$, in the statement of Proposition \ref{prop:finite}  $U'=X$.

 \appendix
\section{Finite maps and base change}\label{subsec:fg}
In this appendix we show that under some conditions, finiteness is open. This fact has been used very crucially in the proof of Proposition \ref{lem:VEA=VA}.
 Openness does not hold in complete generality. See the second answer in \cite{open} for a (non-noetherian) counter example between affine varieties.

\begin{proposition}\label{prop:finite}
    Let $\gamma:S\r T$ be a homomorphism of $\NN$-graded commutative rings,
    where $T$ is finitely generated as $T_0$-algebra and $T_0$ is a finitely generated $S_0$-module.
    Then $\Spec S_0$ contains an open subset $U$ such that for any $S_0$-scheme $\phi:W\r \Spec S_0$
    we have that 
  $W\times_{S_0} \gamma$ is finite if and only if $\im \phi\subset U$.
\end{proposition}
The proof is given below after some lemmas. To reduce verbosity, in the rest of this appendix all rings are commutative.
We use indices to indicate degree restrictions; e.g. if $M$ is $\ZZ$-graded then $M_{[a,b]}$ denotes subset of $M$ whose elements' homogeneous components have degrees in the interval $[a,b]$.
\begin{lemma}{\cite[Proposition 2.2]{ATVdB}}\label{lem:fg1} 
Let $S$ be an $\NN$-graded ring. Let $M$ be a $\ZZ$-graded $S$-module with a left-bounded grading, i.e. $M_{-N}=0$ for $N\gg 0$. Assume that $M_n$ is finitely generated $S_0$-module for every $n\in \ZZ$.
Then  $M$ is finitely generated $S$-module if and only if $M/S_{>0}M$ is right bounded. 
\end{lemma}

\begin{proof}
If $M$ is finitely generated $S$-module, then $M/S_{>0}M$ is always right bounded, while the converse follows by the assumption that $M$ is left-bounded and that $M_n$ is finitely generated $S_0$-module for every $n\in \ZZ$.
\end{proof}
Recall that if $N$ is an $R$-module then the \emph{support} $\Supp_R N$ of $N$ is the set of $P\in \Spec R$ such that $N_P\neq 0$. If $N$ is finitely generated then
$\Supp_R N=V(\Ann_R N)$ and hence in particular $\Supp N$ is closed (see \cite[Lemmas 29.5.3, 29.5.4]{stacks-project}).

For $S$, $M$ as in Lemma \ref{lem:fg1} put $F_S(M)=\bigcap_{N\ge 0} \overline{\Supp_{S_0}(M/S_{>0} M)_{\ge N}}$. Note that $F_S(M)$ is an intersection of closed
sets and hence is itself closed.

\begin{corollary} \label{cor:fg1} Let $S$, $M$ be as in Lemma \ref{lem:fg1}. Then
$M$ is finitely generated $S$-module if and only if $F_S(M)=\emptyset$.
\end{corollary}

\begin{proof}
  The claim follows immediately from Lemma \ref{lem:fg1}, using the definition of $F_S(M)$, and the fact that an affine variety is quasi-compact.
\end{proof}
\begin{lemma} \label{lem:supp} Let $S\r T$ be a morphism between $\NN$-graded algebras such that in addition
  \begin{itemize}
  \item $T_0$ is a finitely generated $S_0$-module.
  \item $T$ is finitely generated in degrees $\le k$ as $T_0$-algebra.
\end{itemize}
Then
\begin{enumerate}
\item $\Supp_{S_0}(T/S_{>0}T)_{\ge N}=\Supp_{S_0}(T/S_{>0}T)_{[N,N+k-1]}$. 
\item $\Supp_{S_0}(T/S_{>0}T)_{\ge N}$ is closed.
  \end{enumerate}
\end{lemma}
\begin{proof} We only need to prove (1) since then (2) follows from the fact that the support of a finitely generated module is closed.
  We clearly have
  \[
  \Supp\nolimits_{S_0}(T/S_{>0}T)_{[N,N+k-1]} \subset \Supp\nolimits_{S_0}(T/S_{>0}T)_{\ge N}.
\]
For the opposite inclusion we note that since $T$ is generated in degree $\le k$, $T_{\ge N}$ is generated in degrees $N,\ldots,N+k-1$ as $T$-module and hence the same holds
for $T/S_{>0}T$. Thus there is surjection
\[
T\otimes_{T_0}  (T/S_{>0}T)_{[N,N+k-1]}\twoheadrightarrow (T/S_{>0}T)_{\ge N}.
\]
Hence if $P\in \Spec S_0$ then $(T/S_{>0}T)_{[N,N+k-1],P}=0$ implies $(T/S_{>0}T)_{\ge N,P}=0$ which yields the opposite inclusion.
\end{proof}

\begin{lemma}\label{lem:fg2}
Let $S$ be an $\NN$-graded ring and let $M$ be a $\ZZ$-graded $S$-module.
Let $S_0\to S_0'$ be a morphism of rings. Put $S'=S_0'\otimes_{S_0} S$, $M'=S_0'\otimes_{S_0} M$. Then $M'/S'_{>0}M'=S_0'\otimes_{S_0}(M/S_{>0}M)$.
\end{lemma}

\begin{proof}
The proof follows by tensoring the right exact sequence
\[
S_{>0}\otimes_{S} M\to M\to M/S_{>0}M\to 0
\]
with the right exact functor $S_0'\otimes_{S_0}(-)$, together with the observation that $S'_{>0}=S_0'\otimes_{S_0} S_{>0}$.
\end{proof}

We recall a standard lemma, see e.g. \cite[Lemma 29.5.3(3)]{stacks-project}.

\begin{lemma}\label{lem_auxiliary}
  Let $\phi:S_0\to S_0'$ be an extension of rings and let $N$ be a
  finitely generated $S_0$-module. Then
  $\Supp_{S'_0}(S_0'\otimes_{S_0}N)=\phi^{-1}(\Supp_{S_0}(N))$.
\end{lemma}

\begin{lemma}\label{lem:fg4}
  Let $S\r T$ be a homomorphism of $\NN$-graded rings,
  where $T$ is finitely generated as $T_0$-algebra and $T_0$ is a finite $S_0$-module.
  Let $\phi:S_0\to S_0'$ be an extension of rings and $S'=S'_0\otimes_{S_0}
  S$, $T'=S'_0\otimes_{S_0} T$. Then $T'$ is finitely generated as
  $S'$-module if and only $\phi^{-1}(F_{S}(T))=\emptyset$.
\end{lemma}

\begin{proof}
Note that $S$ and $M=T$ satisfy the conditions of Lemma \ref{lem:fg1}. The same goes for $S'$ and $M=T'$.
Assume that $T$ is generated as $T_0$-algebra in degrees $\leq k$. Then the same holds for $T'$ viewed as $T'_0$-algebra.
We have by Lemma \ref{lem:fg2}
\begin{equation} \label{eq:tensor}
  T'/S'_{>0} T'=S'_0\otimes_{S_0} T/S_{>0} T.
  \end{equation}
We compute:
\begin{equation}\label{eq:annfgenk}
  \begin{aligned}
    \Supp\nolimits_{S'_0}(T'/S'_{>0}T')_{\ge N}&=\Supp\nolimits_{S'_0}(T'/S'_{>0}T')_{[N,N+k-1]} \qquad &\text{(Lemma \ref{lem:supp}(1))}\\
   &=\Supp\nolimits_{S'_0}(S'_0\otimes_{S_0} (T/S_{>0}T)_{[ N,N+k-1]})\qquad& \text{(by \eqref{eq:tensor})}\\
   & =
  \phi^{-1}( \Supp\nolimits_{S_0}(T/S_{>0}T)_{[N, N+k-1]})\qquad &\text{(Lemma \ref{lem_auxiliary})} \\
&=  \phi^{-1}( \Supp\nolimits_{S_0}(T/S_{>0}T)_{\ge N}) \qquad &\text{(Lemma \ref{lem:supp}(1))}.
  \end{aligned}
\end{equation}
Hence
\begin{align*}
  F_{S'}(T')&=\bigcap_{N\ge 0} \Supp\nolimits_{S'_0} (T'/S'_{>0} T')_{\ge N}\qquad &\text{(Lemma \ref{lem:supp}(2))}\\
            &=\bigcap_{N\ge 0}\phi^{-1}( \Supp\nolimits_{S_0}(T/S_{>0}T)_{\ge N})\qquad &\text{(by \eqref{eq:annfgenk})}\\
            &=\phi^{-1}( \bigcap_{N\ge 0}\Supp\nolimits_{S_0}(T/S_{>0}T)_{\ge N})&\\
  &=\phi^{-1}(F_S(T))\qquad& \text{(Lemma \ref{lem:supp}(2))}.
\end{align*}
 Since by Corollary \ref{cor:fg1}, $T'/S'$ is finite
if and only if $F_{S'}(T')=\emptyset$, we may conclude.
\end{proof}

\begin{proof}[Proof of Proposition \ref{prop:finite}]
  Set $U=\Spec S_0-F_S(T)$. To prove the proposition we may assume that
  $W=\Spec S'_0$ for some ring morphism $\phi:S_0\r S'_0$. We then invoke Lemma \ref{lem:fg4}.
\end{proof}

\bibliographystyle{amsalpha}

\pagebreak
\end{document}